\documentclass[12pt, reqno]{amsart}
\usepackage{amsmath, amsthm, amscd, amsfonts, amssymb, graphicx, color}
\usepackage[bookmarksnumbered, colorlinks, plainpages]{hyperref}
\hypersetup{colorlinks=true,linkcolor=red, anchorcolor=green, citecolor=cyan, urlcolor=red, filecolor=magenta, pdftoolbar=true}

\newtheorem{theorem}{Theorem}[section]
\newtheorem{lemma}[theorem]{Lemma}

\newtheorem{corollary}[theorem]{Corollary}
\theoremstyle{definition}
\newtheorem{definition}[theorem]{Definition}
\newtheorem{example}[theorem]{Example}

\theoremstyle{remark}
\newtheorem{remark}[theorem]{Remark}
\numberwithin{equation}{section}

\begin{document}

\setcounter{page}{1}

\title[$J$-fusion frames]{$J$-fusion frame for Krein spaces}

\author[S. Karmakar]{S. Karmakar$^{*}$}

\address{$^{1}$Department of Mathematics, Jadavpur University, India.}
\email{\textcolor[rgb]{0.00,0.00,0.84}{shibashiskarmakar@gmail.com}}




\subjclass[2010]{Primary 42C15; Secondary 46C05, 46C20.}

\keywords{Krein Space, fusion frames, Gramian operator, uniformly $J$-definite subspaces, $J$-projections.}

\date{\today
\newline \indent $^{*}$Corresponding author}

\begin{abstract}
In this article we introduce the notion of $J$-fusion frame for a Krein space $\mathbb{K}$. We relate this new concept with fusion frames for Hilbert spaces and also with $J$-frames for Krein spaces. We also approximate $J$-fusion frame bounds of a $J$-fusion frame by the upper and lower bounds of the synthesis operator. Finally we address the problem of characterizing those bounded linear operators in $\mathbb{K}$ for which the image of $J$-fusion frame is also a $J$-fusion frame.
\end{abstract} \maketitle

\section{Introduction}
The frame theory was introduced by Duffin and Schaeffer \cite{ds} in the year 1952. But after the work Daubechies et al. \cite{dgm} frame theory developed rapidly. Today frame theory has applications in every modern applied mathematics. Frame theory used in signal processing, image processing, data compression and sampling theory. One of the emerging application of frame theory is to calculate the effect of losses in packet-based communication system and in data transmission. In order to tackle this problems the theory of fusion frames evolved. The idea behind fusion frame is to construct local frames and add them together to get the global frame. Fornasier \cite{mf} used this idea  to quasi-orthogonal subspaces. Casazza et al. \cite{ckl} formulate a general method to introduce fusion frame in Hilbert spaces. Asgari et al. \cite{ak} also worked on fusion frames.

Since fusion frame in Hilbert space has such a huge application so it is a natural demand to extend these ideas in Banach space frame theory and also in Krein space frame theory. Some work already had been done in this direction \cite{kk2008,ws}. In this article we are interested to extend the idea of fusion frame in Krein space. Krein space has some interesting application in modern analysis. The theory of frames in Krein space can be found in \cite{jb,gmmm,efw,ia,hkp}. P. Acosta-Hum$\acute{a}$nez et al. \cite{aef} defined fusion frames in Krein spaces. In their work they found a correspondence between fusion frames in Hilbert spaces and fusion frames in Krein spaces. But their definition involves fundamental symmetry in Krein space which is not unique.

In this article we define fusion frame in Krein spaces in a more geometric setting motivated by the work of Giribet et al. \cite{gmmm}. We also relate this concept with fusion frames for Hilbert spaces and also with $J$-frames for Krein spaces.

\section{Preliminaries and some basic definitions}

In this section we briefly recall some basic notations, definitions and some important properties useful for our further study. For more detailed information we refer the following references \cite{jb,ia,gmmm,ck2003,ak,pg}.
\subsection{Hilbert space frame theory}
A family of vectors $\{f_n\}_{n\in\mathbb{N}}$ is said to be a frame for a Hilbert space $(\mathbb{H},\langle,\rangle)$, if there exists positive real numbers $A$ and $B$ with $0<A \leq B<\infty$ such that
\begin{equation}
 A \| f \|^2 \leq \sum_{i \in I} |\langle f,f_n \rangle|^2  \leq B \| f \|^2
\end{equation}
for all $f \in\mathbb{H}$. $A,B$ are known as lower and upper frame bounds respectively for the frame. If $A = B$ then the frame is known as $A$-tight frame and if $A=B=1$ then the frame
 is known as Parseval frame.\\
 Let $\{f_n\}_{n\in \mathbb{N}}$ be a frame for the Hilbert space $H$ and $\{e_n\}_{n\in \mathbb{N}}$ be the natural orthonormal basis of $\ell_2(\mathbb{N})$. An operator $ T:\mathbb{H} \to \ell_2 (\mathbb{N})$ defined by
 $T(f)=\sum_{n \in \mathbb{N}} \langle f,f_n \rangle e_n$ for all $f\in \mathbb{H}$ is  known as analysis operator and its adjoint operator defined by
 $ T^* (e_n)=f_n $ is known as synthesis operator for the frame $\{f_n\}_{n\in \mathbb{N}}$. The operator $ S (=TT^*):\mathbb{H} \to \mathbb{H} $ given by
 $ S(f)=\sum_{n \in \mathbb{N}} \langle f,f_n \rangle f_n$ for all $f\in \mathbb{H}$ is called frame operator. A mapping $G:\ell_2(\mathbb{N})\to\ell_2(\mathbb{N})$ defined as $G = T^* T$ is known as the Grammian operator.
 It is clear that $S$ is self-adjoint, positive and invertible operator and
 $A.I \leq S \leq B.I$
\begin{definition}
Let $\mathbb{H}$ be a Hilbert space. A sequence $\{f_n\}_{n\in \mathbb{N}}$ is said to ba frame sequence in $\mathbb{H}$ if it a frame for the Hilbert space $\overline{span\{f_n:n\in\mathbb{N}\}}$.
\end{definition}
\subsection{On Krein spaces}
An abstract vector space $(\mathbb{K},[\cdot,\cdot])$ that satisfies the following is called a Krein space.\\

 (i) $\mathbb{K}$ is a linear space over the field $F$, where $F$ is either $\mathbb{R}$ or $\mathbb{C}$.\\

  (ii) there exists a bilinear form $[\cdot,\cdot]\in{F}$ on $\mathbb{K}$ such that\\
  $$[y,x]=\overline{[x,y]}$$
  $$[ax+by,z]=a[x,z]+b[y,z]$$
  for any $x,y,z\in{\mathbb{K}}$, $a,b\in{F}$, where $\overline{[\cdot,\cdot]}$ denote the complex conjugation.\\
	
  (iii) The vector space $\mathbb{K}$ admits a canonical decomposition $\mathbb{K}=\mathbb{K}^+[\dot{+}]\mathbb{K}^-$ such that $(\mathbb{K}^+,[\cdot,\cdot])$ and $(\mathbb{K}^-,-[\cdot,\cdot])$ are Hilbert spaces relative to the norms $\|x\|=[x,x]^{\frac{1}{2}}(x\in{\mathbb{K}^+})$ and $\|x\|=(-[x,x]^{\frac{1}{2}})(x\in{\mathbb{K}^-})$.

Now every canonical decomposition of $\mathbb{K}$ generates two mutually complementary projectors $P_+$ and $P_-$ ($P_{+}+P_-=I$, the identity operator on $\mathbb{K}$ ) mapping $\mathbb{K}$ onto $\mathbb{K}^+$ and $\mathbb{K}^-$ respectively. Thus for any $x\in{\mathbb{K}}$, we have $P_{\pm}=x^{\pm}$, where $x^+\in{\mathbb{K}^+}$ and $x^-\in{\mathbb{K}^-}$. The projectors $P_+$ and $P_-$ are called canonical projectors. The linear operator $J:\mathbb{K}\to{\mathbb{K}}$ defined by the formula $J=P_+-P_-$ is called the canonical symmetry of the Krein space $\mathbb{K}$. The $J$-metric defined by the formula $[x,y]_{J}=[x,Jy]$, where $x,y\in{\mathbb{K}}$. The vector space $\mathbb{K}$ associated with the $J$-metric is a Hilbert space, called the associated Hilbert space of the Krein space $\mathbb{K}$.
\begin{definition}
Let $M$ be a subspace of a Krein space $\mathbb{K}$. $M$ is said to be projectively complete if $\mathbb{K}=M+M^{[\perp]}$.
\end{definition}
The $J$-adjoint of an operator $T$ in Krein spaces, denoted by $T^{\#}$, satisfies $[T(x),y]=[x,T^{\#}(y)]$.
\begin{definition}
A linear operator on a Krein space $\mathbb{K}$ is said to be $J$-selfadjoint if $T=T^{\#}$.
\end{definition}
A $J$-self-adjoint projection is called a $J$-projection.
\begin{definition} \cite{ta}
A subspace $M$ is said to be regular if it is the range of a $J$-projection.
\end{definition}
Every regular subspace is closed. If $M$ is regular with $J$-projection $Q$, then its $J$-orthocomplement $M^{[\perp]}$ with $J$-projection $(I-Q)$ is also regular.

Let $M(\neq\{0\})$ be a subspace of $\mathbb{K}$. Then the operator $G_M$ defined by $G_M=\pi_{M}J|_M$ is the Gram operator of $M$. $\pi_{M}$ is the orthogonal projection from $\mathbb{K}$ onto $M$ (in Hilbert space sense).
\begin{theorem} \cite{ia}
For a subspace $M(\neq\{0\})\subset\mathbb{K}$ with a Gram operator $G_M$ the following concepts are equivalent:\\
$(i)$ $M$ is projectively complete.\\
$(ii)$ $M$ is regular.\\
\end{theorem}
\begin{definition}
Let $(\mathbb{K},[\cdot,\cdot],J)$ be a Krein space. Then a collection of vectors $\mathbb{F}=\{f_n:n\in{\mathbb{N}}\}$ is said to be a Bessel sequence of vectors in $\mathbb{K}$ if it is a Bessel sequence in the associated Hilbert space $(\mathbb{K},[\cdot,\cdot]_J)$ \textit{i.e.} $\sum_{n\in\mathbb{N}}|[f,f_n]_J|^2<\infty$ for all $f\in\mathbb{K}$.
\end{definition}
\subsection{$J$-frames in Krein Spaces}
Let $(\mathbb{K},[\cdot,\cdot],J)$ be a Krein space. Suppose $\mathbb{F}=\{f_n:n\in{\mathbb{N}}\}$ is a Bessel sequence of $\mathbb{K}$ and $T\in{L(\ell^2(I),\mathbb{K})}~(\ell^2(I):=\{(c_i):\sum_{i\in{I}}|c_i|^2<\infty\})$ is the synthesis operator for the Bessel sequence $\mathbb{F}$. Let $I_+=\{i\in{I}:[f_i,f_i]\geq{0}\}$ and $I_-=\{i\in{I}:[f_i,f_i]<0\}$, then $\ell^2(I)=\ell^2(I_+)\oplus{\ell^2(I_-)}$. Also let $P_{\pm}$ denote the orthogonal projection of $\ell^2(I)$ onto $\ell^2(I_{\pm})$. Let $T_{\pm}=TP_{\pm}$, $M_{\pm}=\overline{span}\{ f_i:i\in{I_{\pm}}\}$ then we have $R(T)=R(T_+)+R(T_-)$, where $R(T)$ represents range of the operator $T$.
\begin{definition}\cite{gmmm}
 A Bessel sequence $\mathbb{F}$ is said to be a $J$-frame for $\mathbb{K}$ if $R(T_+)$ is a maximal uniformly $J$-positive subspace of $\mathbb{K}$ and $R(T_-)$ is a maximal uniformly $J$-negative subspace of $\mathbb{K}$.
\end{definition}

\subsection{On Hilbert space fusion frame}
We briefly mention some definitions and results of Hilbert space fusion frame theory. Let $\pi_M$ be the orthogonal projection from the Hilbert space $\mathbb{H}$ onto the subspace $M$ of $\mathbb{H}$. Then the range space of the projection is $M$ \textit{i.e.} $R(\pi_M)=M$ and the null space of this orthogonal projection is $M^{\perp}$ \textit{i.e.} $N(\pi_M)=M^{\perp}$.
\begin{definition}
Let $I$ be some index set and $\{W_i:i\in{I}\}$ be a family of closed subspaces in $\mathbb{H}$. Also let $\{v_i:i\in{I}\}$ be a family of weights \textit{i.e.} $v_i>0~\forall{i\in{I}}$. Then $\{(W_i,v_i):i\in{I}\}$ is a fusion frame if there exist constants $0<C\leq D<\infty$ such that
	\begin{equation}
	C\|f\|^2 \leq \sum_{i\in I} v_i^2|[\pi_{W_i}f,f]| \leq D \|f\|^2 ~ \textmd{~for every $f\in{\mathbb{H}}$}
	\end{equation}  
\end{definition}
$C$ and $D$ are known as lower and upper bounds respectively for the fusion frame. If $C=D$ then the fusion frame is known as $C$-tight fusion frame and if $C=D=1$ then the fusion frame is known as Parseval fusion frame. Moreover, a fusion frame is called $v$-uniform, if $v:=v_i=v_j$ for all $i,~j \in I$. The family of subspaces $\{W_i:i\in{I}\}$ is an orthonormal basis of subspaces if $\mathbb{H}=\dot{\oplus}_{i\in{I}}{W_i}$.
\begin{definition}
A family of subspaces $\{W_i:i\in{I}\}$ of $\mathbb{H}$ is called complete, if $\overline{\sum_{i\in{I}}W_i}=\mathbb{H}$.
\end{definition}
\begin{theorem} \cite{ck2003}
Let $\{(W_i,v_i):i\in{I}\}$ is a fusion frame for $\mathbb{H}$, then it is complete.
\end{theorem}
The converse of the above theorem holds if $\mathbb{H}$ is finite dimensional.

\vspace{0.4 cm}
The following theorem provides a nice interaction between frames in Hilbert spaces and fusion frames in Hilbert spaces.
\begin{theorem}\cite{ck2003}
For each $i\in{I}$, let $v_i>0$ and let $\{f_{ij}\}_{j\in{J_i}}$ be a frame sequence in $\mathbb{H}$ with frame bounds $A_i$ and $B_i$. Define $W_i=\overline{{span}_{j\in{J_i}}\{f_{ij}\}}$ for all $i\in{I}$ and choose an orthonormal basis and $\{e_{ij}\}_{j\in{J_i}}$ for each subspace $W_i$. Suppose that $0<~A=inf_{i\in{I}}A_i~{\leq}~B=sup_{i\in{I}}B_i<~\infty$. The following conditions are equivalent:\\
	
	$(1)$ $\{v_if_{ij}\}_{i\in{I},j\in{J_i}}$ is a frame for $\mathbb{H}$.
	
	$(2)$ $\{v_ie_{ij}\}_{i\in{I},j\in{J_i}}$ is a frame for $\mathbb{H}$.
	
	$(3)$ $\{(W_i,v_i):i\in{I}\}$ is a fusion frame for $\mathbb{H}$.
	
\end{theorem}

For each family of subspaces $\{W_i\}_{i\in{I}}$ of $\mathbb{H}$, we define the space $(\sum_{i\in{I}}\oplus{W_i})_{\ell_2}$ by
$$\big(\sum_{i\in{I}}\oplus{W_i}\big)_{\ell_2}=\big\{\{f_i\}_{i\in{I}}:f_i\in{W_i}~\textmd{and~}\sum_{i\in{i}}\|f_i\|^2<\infty\big\}$$
with inner product given by $\langle\{f_i\}_{i\in{I}},\{g_i\}_{i\in{I}}\rangle=\sum_{i\in{I}}\langle{f_i,g_i}\rangle$.\\

Let $\{(W_i,v_i):i\in{I}\}$ be a fusion frame for $\mathbb{H}$. Then $T_{W,v}:\big(\sum_{i\in{I}}\oplus{W_i}\big)_{\ell_2}\rightarrow\mathbb{H}$ defined by $T_{W,v}(f)=\sum_{i\in{I}}v_if_i~\textmd{for all~}f=\{f_i\}_{i\in{I}}\in\big(\sum_{i\in{I}}\oplus{W_i}\big)_{\ell_2}$ is the synthesis operator for $\{(W_i,v_i):i\in{I}\}$. The adjoint of this operator $T^{\ast}_{W,v}:\mathbb{H}\rightarrow\big(\sum_{i\in{I}}\oplus{W_i}\big)_{\ell_2}$ defined by $T^{\ast}_{W,v}(f)=\{v_i\pi_{W_i}(f)\}_{i\in{I}}$ is the analysis operator for $\{(W_i,v_i):i\in{I}\}$. The fusion frame operator $S_{W,v}$ for $\{(W_i,v_i):i\in{I}\}$ is defined by $S_{W,v}(f)=T_{W,v}T^{\ast}_{W,v}(f)=T_{W,v}(\{v_i\pi_{W_i}(f)\}_{i\in{I}})=\sum_{i\in{I}}v_i^2\pi_{W_i}(f)$. It is a positive, selfadjoint, invertible linear operator on $\mathbb{H}$. 

\section{Main results}
\subsection{On orthogonal and $J$-orthogonal projection on Krein spaces}

Let $\pi_M$ be an orthogonal projection in a Krein space $\mathbb{K}$ onto $M$. Then it an orthogonal projection in Hilbert space sense \textit{i.e.} it is an orthogonal projection in the associated Hilbert space $(\mathbb{K},[\cdot,\cdot]_J)$. Then we have $\pi_M^2=\pi_M$ and $\pi_M^*=\pi_M$. Here range of the projection \textit{i.e.} $R(\pi_M)=M$ and Null space of $\pi_M$ \textit{i.e.} $N(\pi_M)=M^{\perp}$. Let $Q_{M}$ be the $J$-orthogonal projection from $\mathbb{K}$ onto $M$. But the $J$-orthogonal projection $Q_M$ exists if $M$ is a projectively complete subspace of $\mathbb{K}$. Here range of the $J$-projection $Q_M$ \textit{i.e.} $R(Q_M)=M$ and Null space \textit{i.e.} $N(Q_M)=M^{[\perp]}$.

 Let $\pi_M^{\#}$ be the $J$-adjoint of $\pi_M$. Then we have $\pi_M^{\#}=J\pi_M{J}$. Also $\pi_{JM}=J\pi_M{J}$. Hence $\pi_{JM}=\pi_M^{\#}$. 

Let $W$ be a closed subspace of $M$, where $M$ is an uniformly $J$-definite subspace of $\mathbb{K}$. Then $W$ is a regular subspace of $\mathbb{K}$. Hence the $J$-orthogonal projection from $\mathbb{K}$ onto $W$ exists, let it be $Q_W$.
\begin{lemma}\label{RJPP}
Let $M$ be uniformly $J$-definite subspace for the Krein space $\mathbb{K}$. If $W$ is a closed subspace of $M$ then ${Q_{W}|}_{M}={\pi_{W}|}_{M}$.
\end{lemma}
\begin{proof}
Let $x\in{M}$ and $x_1,x_2\in{W}$.
Let ${Q_{W}|}_{M}(x)=x_1$ and ${\pi_{W}|}_{M}(x)=x_2$. Then $x-x_1 [\perp] W$ which implies that $[x-x_1,w]=0~\forall{w\in{W}}$. Since $M$ is uniformly $J$-definite hence $[x-x_1,w]_J=0~\forall{w\in{W}}$. Then by the definition of orthogonal projection we have ${\pi_{W}|}_{M}(x)=x_1$. So we have $x_1=x_2$. Hence ${Q_{W}|}_{M}={\pi_{W}|}_{M}$.
\end{proof}
Since $(M,[\cdot,\cdot])$ is itself a Hilbert space. So let $P_W$ be the orthogonal projection from $M$ onto $W$. The above lemma states that $P_W={Q_{W}|}_{M}={\pi_{W}|}_{M}$.

\vspace{0.5 cm}
Let $W$ be a subspace of a Krein space $\mathbb{K}$. Let $\mathbb{P}^{++}$ denote the set of all positive subspaces of $\mathbb{K}$, $\mathbb{P}^{+}$ denote the set of all non-negative subspaces of $\mathbb{K}$. Similarly $\mathbb{P}^{--}$ and $\mathbb{P}^{-}$ denote the set of all negative and non-positive subspaces of $\mathbb{K}$ respectively. Also let $\tilde{\mathbb{P}}$ be the set of all indefinite subspaces of $\mathbb{K}$. The set of all neutral subspaces sometimes referred as $\mathbb{P}^{0}$. We have $\mathbb{P}^{0}\subset\mathbb{P}^{+}$ or $\mathbb{P}^{0}\subset\mathbb{P}^{-}$. Then $W\in{\mathbb{P}^{+}\cup\mathbb{P}^{-}\cup\tilde{\mathbb{P}}}$. Throughout in our work we consider either $W\in{\mathbb{P}^{+}\cup\mathbb{P}^{--}}$ or $W\in{\mathbb{P}^{++}\cup\mathbb{P}^{-}}$. Without any loss of generality we assume $W\in{\mathbb{P}^{+}\cup\mathbb{P}^{--}}$ to establish our results.

Let $\{W_i:i\in I\})$ be a collection of subspaces of the Krein space $\mathbb{K}$ such that $W_i\in{\mathbb{P}^{+}\cup\mathbb{P}^{--}}~\forall{i\in I}$. We consider the space $\big(\sum_{i\in{I}}\oplus{W_i}\big)$. Then if $f\in\big(\sum_{i\in{I}}\oplus{W_i}\big)$ then $f=\{f_i\}_{i\in I}$, where $f_i\in W_i$ for each $i\in I$. Let $I_+=\{i\in{I}:[f_i,f_i]\geq 0~\textmd{for all~} f_i\in{W_i}\}$ and $I_-=\{i\in{I}:[f_i,f_i]<0~\textmd{for all~} f_i\in{W_i}\}$. We define $[f,g]=\sum_{i\in{I}}[f_i,g_i]$, where $f,g\in\big(\sum_{i\in{I}}\oplus{W_i}\big)$. If the series is unconditionally convergent then $[\cdot,\cdot]$ defines an inner product on $\big(\sum_{i\in{I}}\oplus{W_i}\big)$. 

\begin{definition}
Let $I$ be some index set and let $\{v_i:i\in{I}\}$ be a family weights \textit{i.e.} $v_i>0~\forall~{i\in{I}}$. Let $\{W_i:i\in{I}\}$ be a family of subspaces of a Krein space $\mathbb{K}$. Then the collection $\mathbb{F}=\{(W_i,v_i):i\in{I}\}$ is said to be a Bessel family of subspaces if $\sum_{i\in{I}}v_i^2[\pi_{W_i}(f),f]_J\leq C[f,f]_J$ for all $f\in \mathbb{K}$
\end{definition}
\subsection{Definition of $J$-fusion frames in Krein space} 
Let $\mathbb{F}=\{(W_i,v_i):i\in{I}\}$ be a Bessel family of closed subspaces of a Krein space $\mathbb{K}$ with synthesis operator $T_{W,v}\in{L\Big(\big(\sum_{i\in{I}}\oplus{W_i}\big)_{\ell_2},\mathbb{K}\Big)}$ such that $W_i\in{\mathbb{P}^{+}\cup\mathbb{P}^{--}}~\forall{i\in I}$. Let $I_+=\{i\in{I}:[f_i,f_i]\geq 0~\textmd{for all~} f_i\in{W_i}\}$ and $I_-=\{i\in{I}:[f_i,f_i]<0~\textmd{for all~} f_i\in{W_i}\}$. Now consider the orthogonal decomposition of $\big(\sum_{i\in{I}}\oplus{W_i}\big)_{\ell_2}$ given by
$$\big(\sum_{i\in{I}}\oplus{W_i}\big)_{\ell_2}=\big(\sum_{i\in{I_+}}\oplus{W_i}\big)_{\ell_2}\bigoplus{\big(\sum_{i\in{I_-}}\oplus{W_i}\big)_{\ell_2}},$$
and denote by $P_{\pm}$ the orthogonal projection onto $(\sum_{i\in{I_{\pm}}}\oplus{W_i})_{\ell_2}$. Also, let ${T_{W,v}}_{\pm}=T_{W,v}P_{\pm}$. If $M_{\pm}=\overline{\sum_{i\in{I_{\pm}}}W_i}$, notice that $\sum_{i\in{I_{\pm}}}W_i\subseteq{R({T_{W,v}}_{\pm})}\subseteq{M_{\pm}}$ and
$$R(T_{W,v})=R({T_{W,v}}_+)+R({T_{W,v}}_-)$$

The Bessel family $\mathbb{F}=\{(W_i,v_i):i\in{I}\}$ is a $J$-fusion frame for $\mathbb{K}$ if $R({T_{W,v}}_+)$ is a maximal uniformly $J$-positive subspace of $\mathbb{K}$ and $R({T_{W,v}}_-)$ is a maximal uniformly $J$-negative subspace of $\mathbb{K}$.

Let $\{(W_i,v_i):i\in{I}\}$ is a $J$-fusion frame for $\mathbb{K}$ then $\big(\sum_{i\in{I}}\oplus{W_i},[\cdot,\cdot]\big)$ is a Krein space. The fundamental symmetry let $J_2$ is defined by $J_2(f)=\{f_i:i\in I_+\}\cup\{-f_i:i\in I_-\}$ for all $f$. Also $[f,g]_{J_2}=\sum_{i\in{I_+}}[f_i,g_i]-\sum_{i\in{I_-}}[f_i,g_i]$. Now consider the space $\big(\sum_{i\in{I}}\oplus{W_i}\big)_{\ell_2}=\Big\{f:\big(\sum_{i\in{I}}\oplus{W_i}\big):\sum_{i\in{I}}\|f_i\|_J^2<\infty\Big\}$. We will use this space frequently in our work.

\vspace{0.4 cm}
Let $\mathbb{H}$ be a finite dimensional Hilbert space. Then from the theory of Hilbert space fusion frame we know that any complete family of subspaces is a fusion frame for $\mathbb{H}$ with respect to some arbitrary weights $\{v_i\}$. But similar to the $J$-frame theory of Krein spaces, any complete family of non-neutral subspaces in a finite dimensional Krein space may not be a $J$-fusion frame for that Krein space.
\begin{example}
Consider the vector space $\mathbb{R}^3(\mathbb{R})$. Let $\{e_1,e_2,e_3\}$ be the standard orthonormal basis for $\mathbb{R}^3$. Define $[e_1,e_1]=1,~[e_2,e_2]=1$ and $[e_3,e_3]=-1$. Also $[e_i,e_j]=0$ for $i\neq{j},~i,j\in\{1,2,3\}$. Then $(\mathbb{R}^3,[\cdot,\cdot])$ is a Krein space.

Now let $W=\{W_1,W_2,W_3\}$ is a family of uniformly $J$-definite subspace of $\mathbb{R}^3$, where $W_1=span\{e_1+\frac{1}{\sqrt2}e_3\},~W_2=span\{e_2+\frac{1}{\sqrt2}e_3\}~\textmd{and }W_3=span\{e_3\}$. Then for any family of weights $\{v_i:i\in\{1,2,3\}\}$, the collection $W$ is not a $J$-fusion frame of $\mathbb{R}^3$. But it is a fusion frame for $(\mathbb{R}^3,[\cdot,\cdot]_J)$, considered as a Hilbert spaces.
\end{example}
If $\mathbb{F}=\{(W_i,v_i):i\in{I}\}$ is a $J$-fusion frame for $\mathbb{K}$ then $R({T_{W,v}}_+)$ is a maximal uniformly $J$-positive subspace of $\mathbb{K}$ and $R({T_{W,v}}_-)$ is a maximal uniformly $J$-negative subspace of $\mathbb{K}$. So $R({T_{W,v}}_{\pm})$ is closed. Hence $R({T_{W,v}}_{\pm})=M_{\pm}$. Then we have $\mathbb{K}=M_+\oplus{M_-}$. Then $\mathbb{F}=\{(W_i,v_i):i\in{I}\}$ is a fusion frame in the associated Hilbert space $(\mathbb{K},[\cdot,\cdot]_J)$. So in terms of frame inequality we have
\begin{equation*}
A\|f\|_J^2 \leq \sum_{i\in{I}}v_i^2\|\pi_{W_i}(f)\|_J^2 \leq A\|f\|_J^2 ~~\textmd{for all~}~f\in \mathbb{K},
\end{equation*}
where $0<A\leq B<\infty$. Let $T_{{W,v}_+}^{*}$ be the analysis operator for the above fusion frame then $T_{W,v}^{*}(f)=\{v_i\pi_{W_i}(f)\}_{i\in{I}}$ for all $f\in{\mathbb{K}}$.

The following theorem is a generalization of a theorem in \cite{gmmm} in $J$-fusion frame setting.
\begin{theorem}
Let $\mathbb{F}=\{(W_i,v_i)\}_{i\in{I}}$ be a $J$-fusion frame for $\mathbb{K}$. Then $\mathbb{F}_{\pm}=\{(W_i,v_i)\}_{i\in{I}_{\pm}}$ is fusion frame for the Hilbert space $(M_{\pm},\pm[\cdot,\cdot])$ \textit{i.e.} there exists constants $B_-\leq{A_-}<0<A_+\leq{B_+}$ such that
\begin{equation}\label{JFFEQ}
A_{\pm}[f,f]\leq\sum_{i\in{I_{\pm}}}v_i^2|[{\pi_{W_i}|}_{M_\pm}(f),f]|~{\leq}~B_{\pm}[f,f]~\textmd{for every }f\in{M_\pm}
\end{equation} 
\end{theorem}
\begin{proof}
It is clear that $\mathbb{F}_{\pm}=\{(W_i,v_i)\}_{i\in{I}_{\pm}}$ is fusion frame for the Hilbert space $(M_{\pm},\pm[\cdot,\cdot])$. Without any loss of generality we only considering the positive part. So we have, $A_{\pm}[f,f]\leq\sum_{i\in{I_{+}}}|[v_iP_{W_i}(f),v_iP_{W_i}(f)]|~{\leq}~B_{\pm}[f,f]$ for all $f\in{M_\pm}$, where $P_{W_i}$ is the orthogonal projection from the Hilbert space $(M_\pm,\pm[\cdot,\cdot])$ onto $W_i$. Now from lemma (\ref{RJPP}) we have $[P_{W_i}(f),f]=[{\pi_{W_i}|}_{M_\pm}(f),f]$.
So we are done.
\end{proof}
\begin{remark}
Let $\{f_i:i\in{I}\}$ be a $J$-frame for the Krein space $\mathbb{K}$. We assume that $M$ is a uniformly $J$-definite subspace of $\mathbb{K}$. Then we know that $\{\pi_M(f_i):i\in{I}\}$ is also a $J$-frame for the subspace $M$. But since $M$ is uniformly $J$-definite hence $(M,[\cdot,\cdot])$ is a Hilbert space. So $\{\pi_M(f_i):i\in{I}\}$ is also a frame for the subspace $M$ in Hilbert space sense.
\end{remark}
\begin{definition}
A sequence $\{f_i:i\in{I}\}$ in $\mathbb{K}$ is said to be a $J$-frame sequence in $\mathbb{K}$ if $\overline{span\{f_i:i\in{I_+}\}}$ and $\overline{span\{f_i:i\in{I_-}\}}$ are uniformly $J$-positive and uniformly $J$-negative subspace of $\mathbb{K}$ respectively.
\end{definition}
Let $\mathbb{F}=\{f_i:i\in{I}\}$ be a $J$-frame for $\mathbb{K}$ then there exists constants $B_-$, $A_-$, $A_+$ and $B_+$ such that $-\infty<B_-\leq A_-<0<A_+\leq B_+<\infty$. These constants are the $J$-frame bounds of the $J$-frame $\mathbb{F}$. Now for every $J$-frame sequence such constants exists. But if the $J$-frame sequence consists of only positive or negative element respectively then the positive or negative parts of the constants redundant accordingly.
\begin{theorem}
	For each $i\in{I}$, let $v_i>0$ and let $\{f_{ij}\}_{j\in{J_i}}$ be a $J$-frame sequence in $\mathbb{K}$ with $J$-frame bounds $B_{i-}$, $A_{i-}$, $A_{i+}$ and $B_{i+}$ for each $i\in{I}$ such that $-\infty<sup_{i}B_{i-}~{\leq}~inf_{i}A_{i-}<0<inf_{i}A_{i+}~{\leq}~sup_{i}B_{i+}<\infty$. Define $W_i=\overline{{span}_{j\in{J_i}}\{f_{ij}\}}$ for each $i\in{I}$. If $W_i$ is definite for each $i\in{I}$ then the following conditions are equivalent:\\
	
	$(i)$ $\{v_if_{ij}\}_{i\in{I},j\in{J_i}}$ is a $J$-frame for $\mathbb{K}$.\\
	
	$(ii)$ $\{(W_i,v_i):i\in{I}\}$ is a $J$-fusion frame for $\mathbb{K}$.
	
\end{theorem}
\begin{proof}
Let $\{v_if_{ij}\}_{i\in{I_+},j\in{J_i}}$ be a $J$-frame for $\mathbb{K}$. So let $I_+=\{(i,j):[f_{ij},f_{ij}]>0\}$ and $I_-=\{(i,j):[f_{ij},f_{ij}]<0\}$. Also let $M_+=\overline{span\{f_{ij}:(i,j)\in{I_+}\}}$ and $M_-=\overline{span\{f_{ij}:(i,j)\in{I_-}\}}$. Also we have
\begin{equation*}
A_{\pm}[f,f]\leq \sum_{(i,j)\in{I_{\pm}}}|[f,v_if_{ij}]|^2\leq B_{\pm}[f,f]~\textmd{for all }~f\in{M_{\pm}}
\end{equation*}
Now $\{f_{ij}\}_{j\in{J_i}}$ is a $J$-frame sequence in $\mathbb{K}$ and $W_i=\overline{span\{f_{ij}\}_{j\in{J_i}}}$. But $W_i$ is a definite subspace of $\mathbb{K}$, so either $W_i\subset{M_+}$ or $W_i\subset{M_-}$ for each $i$. Hence $\{f_{ij}\}_{j\in{J_i}}$ is a $J$-frame for $\mathbb{K}$. Here these $J$-frames either consists of only positive elements or only negative elements. Without any loss of generality we choose the pair $(i,j)$ s.t. $(i,j)\in{I_+}$, then $W_i\subset{M_+}$. Then,
\begin{equation*}
A_{i+}[\pi_{W_i}|_{M_+}(f),f]\leq \sum_{j\in{J_i}}|[\pi_{W_i}f,f_{ij}]|^2\leq B_{i+}[\pi_{W_i}|_{M_+}(f),f]~\forall~f\in{M_+}
\end{equation*}
Now let us assume $0<~A_1=inf_{i}A_{i+}~{\leq}~B_1=sup_{i}B_{i+}<~\infty$, then 
\begin{equation*}
\begin{split}
A_1\sum_{i}v_i^2[\pi_{W_i}|_{M_+}(f),f] & {\leq} \sum_{i}A_{i+}v_i^2[\pi_{W_i}|_{M_+}(f),f]\\
& {\leq} \sum_{i}\sum_{j\in{J_i}}|[\pi_{W_i}|_{M_+}(f),v_if_{ij}]| ^2\\
& {\leq} \sum_{i}B_{i+}[\pi_{W_i}|_{M_+}(f),f]\\
& {\leq} B_1\sum_{i}v_i^2[\pi_{W_i}|_{M_+}(f),f]
\end{split}
\end{equation*}
Now we observe that $\sum_{i}\sum_{j\in{J_i}}|[{\pi_{W_i}|_{M_+}(f),v_if_{ij}}]|^2=\sum_{i}\sum_{j\in{J_i}}|[{f,v_if_{ij}}]|^2$.\\
Since $[\pi_{W_i}|_{M_+}(f),f]>0$ for all $f\in{M_+}$, so
\begin{equation*}
\begin{split}
A_1\sum_{i}v_i^2[\pi_{W_i}|_{M_+}(f),f] & {\leq} \sum_{i}\sum_{j\in{J_i}}|[{f,v_if_{ij}}]|^2\\
& {\leq} B_+[f,f]
\end{split}
\end{equation*}
So we have $\sum_{i}v_i^2|[\pi_{W_i}|_{M_+}(f),f]|\leq \frac{B_+}{A_1}[f,f]$. 

Similarly we can show that $\sum_{i}v_i^2|[\pi_{W_i}|_{M_+}(f),f,J(f)]|
\leq \frac{A_+}{B_1}[f,f]$.\\
When $W_i\subset{M_-}$ for some $i$. Then we can proceed as above. This proves that $\{(W_i,v_i):i\in{I}\}$ is a $J$-fusion frame of subspaces for $\mathbb{K}$. A careful investigation of the above implication reveals that the implications are vice versa. So we prove that $(1)\Leftrightarrow(2)$.
\end{proof}
Now let $\mathbb{F}=\{(W_i,v_i):i\in{I}\}$ be a $J$-fusion frame for the Krein space $\mathbb{K}$. Then $\{W_i:i\in{I_+}\}$ is a collection of uniformly $J$-positive subspaces of $\mathbb{K}$ and $\{W_i:i\in{I_-}\}$ is a collection of uniformly $J$-negative subspaces of $\mathbb{K}$. Let $T_{W,v}^{\#}$ be the $J$-adjoint operator of the synthesis operator $T_{W,v}$. $T_{W,v}^{\#}$ is called the analysis operator of the $J$-frame of subspaces $F$. Now $T_{W,v}^{\#}=(T_{{W,v}_+}^{\#}+T_{{W,v}_-}^{\#})$. Now $N(T_{{W,v}_+}^{\#})^{[\perp]}=R(T_{{W,v}_+})=M_+$. We want to calculate $T_{{W,v}_+}^{\#}$. 
\begin{eqnarray*}
[T_{{W,v}_+}^{\#}(f),\{g_i\}_{i\in{I_+}}] & =[f,T_{{W,v}_+}(\{g_i\}_{i\in{I_+}})] & =[f,\sum_{i\in{I_+}}v_ig_i]\\
& =\sum_{i\in{I_+}}v_i[f,g_i] & =\sum_{i\in{I_+}}v_i[J(f),g_i]_J\\
& =\sum_{i\in{I_+}}v_i[J(f),\pi_{W_i}(g_i)]_J & =\sum_{i\in{I_+}}v_i[\pi_{W_i}J(f),g_i]_J\\
& =\sum_{i\in{I_+}}v_i[J\pi_{W_i}J(f),g_i] & =\sum_{i\in{I_+}}v_i[\pi_{JW_i}(f),g_i]
\end{eqnarray*}
In the above calculation note that $[f,g_i]=[f,{Q_{W_i}|}_{M_+}(g_i)]=[f,Q_{W_i}(g_i)]=[Q_{W_i}(f),g_i]$. Hence $T_{{W,v}_+}^{\#}(f)=\{v_iQ_{W_i}(f)\}_{i\in{I_+}}$. Since $Q_{W_i}(f)\in{W_i}$, hence $T_{{W,v}_+}^{\#}\in\big(\sum_{i\in{I_+}}\oplus{W_i}\big)_{\ell_2}$.

Now we have $T_{{W,v}_+}^{\#}(f)=\{v_i\pi_{JW_i}(f)\}_{i\in{I_+}}$ for all $f\in{\mathbb{K}}$. By using similar arguments as above we have $T_{{W,v}_-}^{\#}(f)=-\{v_i\pi_{JW_i}(f)\}_{i\in{I_-}}$ for all $f\in{\mathbb{K}}$. So $T_{W,v}^{\#}(f)=\{\sigma_iv_i\pi_{JW_i}(f)\}_{i\in{I}}$ for all $f\in{\mathbb{K}}$. Also note that $T_{{W,v}_+}^{\#}(f)=\{v_i\pi_{W_i}(f)\}_{i\in{I_+}}$ for all $f\in{M_+}$.
Here $\sigma_{i}=1$ if $i\in{I_+}$ and $\sigma_{i}=-1$ if $i\in{I_-}$.
\begin{corollary}\label{RPS}
For all $f\in{\mathbb{K}}$, we have $\pi_{JW_i}\pi_{M_+}=\pi_{W_i}$.
\end{corollary}
\begin{lemma}
Let $\{(W_i,v_i):i\in{I}\}$ be a $J$-fusion frame for the Krein space $\mathbb{K}$. Then $\{(J(W_i),v_i):i\in{I}\}$ is also a $J$-fusion frame for $\mathbb{K}$.
\end{lemma}
\begin{proof}
Let $I_+=\{i\in{I}:[f_i,f_i]\geq 0~\textmd{for all~} f_i\in{W_i}\}$ and $I_+=\{i\in{I}:[f_i,f_i]<0~\textmd{for all~} f_i\in{W_i}\}$. Then $W_i\subset{M_+}$ for all $i\in{I_+}$ and $W_i\subset{M_-}$ for all $i\in{I_-}$. Then we have $J(W_i)\subset{J(M_+)}$ for all $i\in{I_+}$ and $J(W_i)\subset{J(M_-)}$ for all $i\in{I_-}$. Since $\{(W_i,v_i):i\in{I}\}$ be a $J$-fusion frame for $\mathbb{K}$ so $M_+=\overline{\sum_{i\in{I_+}}W_i}$ and $M_-=\overline{\sum_{i\in{I_-}}W_i}$. From this it readily follows that $J(M_+)=\overline{\sum_{i\in{I_+}}J(W_i)}$ and $J(M_-)=\overline{\sum_{i\in{I_-}}J(W_i)}$. Now let $f\in{J(M_+)}$, then $J(f)\in{M_+}$. So 
\begin{eqnarray*}
[\pi_{W_i}J(f),J(f)] &=[\pi_{W_i}J(f),f]_J &=[J\pi_{W_i}J(f),f]\\
&=[\pi_{JW_i}(f),f]
\end{eqnarray*}
In terms of the inequality (\ref{JFFEQ}) we have
\begin{equation*}
A_{\pm}[f,f]\leq\sum_{i\in{I_{\pm}}}v_i^2|[{\pi_{J(W_i)}|}_{J(M_\pm)}(f),f]|~{\leq}~B_{\pm}[f,f]~\textmd{for every }f\in{J(M_\pm)}
\end{equation*}
\end{proof}
Let $T_{J(W),v}$ be the synthesis operator for the $J$-fusion frame $\{(J(W_i),v_i):i\in{I}\}$. By using similar arguments we also have $T_{{J(W),v}_+}^{\#}(f)=\{v_i\pi_{W_i}(f)\}_{i\in{I_+}}$ for all $f\in{\mathbb{K}}$. So $T_{J(W),v}^{\#}(f)=\{\sigma_iv_i\pi_{W_i}(f)\}_{i\in{I}}$ for all $f\in{\mathbb{K}}$. Also note that $T_{{J(W),v}_+}^{\#}(f)=\{v_i\pi_{J(W_i)}(f)\}_{i\in{I_+}}$ for all $f\in{J(M_+)}$.
\begin{corollary}
For all $f\in{\mathbb{K}}$, we have $\pi_{W_i}\pi_{J(M_+)}=\pi_{J(W_i)}$.
\end{corollary}

\subsection{Bounds of $J$-fusion frame}
\begin{definition}
Let $\mathbb{F}=\{(W_i,v_i):i\in{I}\}$ be a $J$-fusion frame for $\mathbb{K}$ then there exists constants $B_-$, $A_-$, $A_+$ and $B_+$ such that $-\infty<B_-\leq A_-<0<A_+\leq B_+<\infty$. These constants are the $J$-fusion frame bounds of the $J$-fusion frame $\mathbb{F}$ in $\mathbb{K}$. If these bounds are optimal then they are called optimal $J$-fusion frame bound.
\end{definition}
\begin{definition}
The reduced minimum modulus $\gamma(T)$ of an operator $T\in{L(\mathbb{H},\mathbb{K})}$ is defined by
\begin{equation*}
\gamma(T)=inf\{\|Tx\|:x\in N(T )^{\perp}, \|x\|=1\}.
\end{equation*}
\end{definition}
It is well known that $\gamma(T)=\gamma(T^\ast)=\gamma(TT^\ast)^{\frac{1}{2}}$.

\vspace{0.2 cm}
We want to calculate $J$-fusion frame bounds of a $J$-fusion frame in a Krein space. Let $F=\{(W_i,v_i):i\in{I}\}$ be a $J$-frame of subspaces for the Krein space $\mathbb{K}$. Then for all $f\in{M_+}$ we have
\begin{equation*}
\begin{split}
\sum_{i\in{I_+}}v_i^2[\pi_{W_i}|_{M_+}(f),f] & =\sum_{i\in{I_+}}v_i^2[\pi_{W_i}(f),\pi_{W_i}(f)]\\
=\|T_{{W,v}_+}^{\#}(f)\|_{J_2}^2 & \leq \|T_{{W,v}_+}^{\#}\|_{J_2}^2\|f\|_J^2\\
& \leq \frac{1}{\gamma(G_{M_+})}\|T_{{W,v}_+}\|_J^2~[f,f]
\end{split}
\end{equation*}
Comparing with equation (\ref{JFFEQ}) we have $B_+=\frac{1}{\gamma(G_{M_+})}\|T_{{W,v}_+}\|_J^2$. Again for all $f\in{M_+}$,
\begin{equation*}
\sum_{i\in{I_+}}v_i^2[\pi_{W_i}|_{M_+}(f),f] =\sum_{i\in{I_+}}v_i^2[\pi_{W_i}(f),\pi_{W_i}(f)] =\|T_{{W,v}_+}^{\#}(f)\|_{J_2}^2
\end{equation*}
Now if $f\in{\mathbb{K}}$, then $T_{{W,v}_+}^{\#}(f)=\{v_i\pi_{J(W_i)}(f)\}_{i\in{I_+}}$. So,
$$\|T_{{W,v}_+}^{\#}(f)\|_{J_2}^2=\sum_{i\in{I_+}}v_i^2[\pi_{W_i}J(f),\pi_{W_i}J(f)]$$
For any $f\in{\mathbb{K}}$, $\pi_{M_+}J(f)\in{M_+}$, so $T_{{W,v}_+}^{\#}(\pi_{M_+}J(f))=\{v_i\pi_{J(W_i)}\pi_{M_+}J(f)\}_{i\in{I_+}}$. Using corollary (\ref{RPS}), we have $T_{{W,v}_+}^{\#}(\pi_{M_+}J(f))=\{v_i\pi_{W_i}J(f)\}$. Hence,
\begin{equation*}
\begin{split}
\|T_{{W,v}_+}^{\#}(f)\|_{J_2}^2 & =\|T_{{W,v}_+}^{\#}\pi_{M_+}J(f)\|_{J_2}^2\\
& \geq \gamma(T_{{W,v}_+}^{\#})^2\|\pi_{M_+}J(f)\|_J^2\\
& =\gamma(T_{{W,v}_+})^2\|\pi_{M_+}J(f)\|_J^2\\
& =\gamma(T_{{W,v}_+})^2\|G_{M_+}(f)\|_J^2\\
& \geq \gamma(T_{{W,v}_+})^2\gamma(G_{M_+})^2\|f\|_J^2\\
& \geq \gamma(T_{{W,v}_+})^2\gamma(G_{M_+})^2~[f,f].
\end{split}
\end{equation*}
Comparing with equation (\ref{JFFEQ}) we have $A_+=\gamma(T_{{W,v}_+})^2\gamma(G_{M_+})^2$. Similarly we have $B_-=-\frac{1}{\gamma(G_{M_-})}
\|T_{{W,v}_-}\|_J^2$ and $A_-=-\gamma(T_{{W,v}_-})^2\gamma(G_{M_-})^2$.

Here of course the $J$-fusion frame bounds calculated above are not optimal.

Let $\{(W_i,v_i):i\in{I}\}$ be a $J$-fusion frame of subspaces for the Krein space $\mathbb{K}$. Then according to our definition $M_+=\overline{\sum_{i\in{I_+}}W_i}$ and $M_-=\overline{\sum_{i\in{I_-}}W_i}$. Then $M_+$ and $M_-$ are closed uniformly $J$-positive and $J$-negative subspace respectively. Now since $(\mathbb{K},[\cdot,\cdot],J)$ be a Krein space. So let $\mathbb{K}=\mathbb{K}^{+}{[\dot{\oplus}]}\mathbb{K}^-$ be the cannonical decomposition of $\mathbb{K}$. Let $K$ be the angular operator of $M_+$ with respect to $\mathbb{K}^+$. Then $\|K\|=\frac{1-\gamma(G_{M_+})}{1+\gamma(G_{M_+})}$ and also the domain of definition of $K$ is $\mathbb{K}^+$. Similarly let $Q$ be the angular operator of $M_-$ with respect to $\mathbb{K}^-$. Then $\|Q\|=\frac{1-\gamma(G_{M_-})}{1+\gamma(G_{M_-})}$ and also the domain of definition of $Q$ is $\mathbb{K}^-$.

\section{The $J$-fusion frame operator}
Let $\{(W_i,v_i\}_{i\in I}$ be a $J$-fusion frame for the Krein space $\mathbb{K}$. Then $I_+=\{i\in{I}:[f_i,f_i]\geq 0~\textmd{for all~} f_i\in{W_i}\}$ and $I_+=\{i\in{I}:[f_i,f_i]<0~\textmd{for all~} f_i\in{W_i}\}$.
\begin{definition}
The linear operator $S_{W,v}:\mathbb{K}\rightarrow{\mathbb{K}}$ defined by $S_{W,v}(f)=\sum_{i\in I}\sigma_iv_i^2\pi_{J(W_i)}(f)$ is said to be the $J$-fusion frame operator for the $J$-fusion frame $\{(W_i,v_i\}_{i\in I}$. Here $\sigma_{i}=1$ if $i\in{I_+}$ and $\sigma_{i}=-1$ if $i\in{I_-}$.
\end{definition}
From the definition it readily follows that $S_{W,v}=T_{W,v}T_{W,v}^{\#}$. Also $S_{W,v}$ is the sum of two $J$-positive operators. The statement also easily follows from the definition. Let $S_{{W,v}_+}:\mathbb{K}\rightarrow{\mathbb{K}}$ is defined by $S_{{W,v}_+}(f)=\sum_{i\in I_+}v_i^2\pi_{J(W_i)}(f)$. Then $[S_{{W,v}_+}(f),f]=[\sum_{i\in I_+}v_i^2\pi_{J(W_i)}(f),f]=\sum_{i\in I_+}v_i^2[\pi_{W_i}(f),f]_J$. So $S_{{W,v}_+}$ is a $J$-positive operator and also $S_{{W,v}_+}=T_{{W,v}_+}T_{{W,v}_+}^{\#}$. Similarly let $S_{{W,v}_-}=-T_{{W,v}_-}T_{{W,v}_-}^{\#}$. Then $S_{{W,v}_-}$ is also a $J$-positive operator. We also have $S_{W,v}=S_{{W,v}_+}-S_{{W,v}_-}$.
\begin{theorem}
If $\{(W_i,v_i)\}_{i\in I}$ is a $J$-fusion frame for the Krein space $\mathbb{K}$ wth synthesis operator $T_{W,v}\in{L\Big(\big(\sum_{i\in{I}}\oplus{W_i}\big)_{\ell_2},\mathbb{K}\Big)}$ then the $J$-fusion frame operator $S_{W,v}$ is bijective and $J$-selfadjoint.
\end{theorem}
\begin{proof}
Let $S_{W,v}(f)=0$ then $S_{{W,v}_+}(f)=S_{{W,v}_-}(f)$. Again $R(S_{{W,v}_+})=M_+$ and $R(S_{{W,v}_-})=M_-$. But $M_+\cap M_-=\{0\}$. These together implies that $f=0$. Hence $S_{W,v}$ is injective. To prove the surjectivity of $S_{W,v}$ we consider the oblique decomposition of $\mathbb{K}$. We can write $\mathbb{K}=M_+^{[\perp]}\oplus M_-^{[\perp]}$. Now $N(S_{{W,v}_{\pm}})=M_{\pm}^{[\perp]}$. So $S_{W,v}(M_{\pm}^{[\perp]})=S_{{W,v}_{\mp}}(M_{\pm}^{[\perp]})$ and $R(S_{W,v})=S_{{W,v}_{-}}(M_{+}^{[\perp]})+S_{{W,v}_{+}}(M_{-}^{[\perp]})=R(S_{{W,v}_+})+R(S_{{W,v}_-})=M_++M_+=\mathbb{K}$. Therefore $S_{W,v}$ is bijective. Since $S_{W,v}=T_{W,v}T_{W,v}^{\#}$ hence $S_{W,v}$ is $J$-selfadjoint.
\end{proof}
\begin{theorem}
If $\{(W_i,v_i)\}_{i\in I}$ is a $J$-fusion frame for the Krein space $\mathbb{K}$ with $J$-fusion frame operator $S_{W,v}$, then $\{(S_{W,v}^{-1}(W_i),v_i)\}_{i\in I}$ is a $J$-fusion frame for $\mathbb{K}$ with $J$-fusion frame operator $S_{W,v}^{-1}$.
\end{theorem}
\begin{proof}
We observe that the operator $S_{W,v}^{-1}T_{W,v}\in{L\Big(\big(\sum_{i\in{I}}\oplus{W_i}\big)_{\ell_2},\mathbb{K}\Big)}$ is the synthesis operator of $\{(S_{W,v}^{-1}(W_i),v_i)\}_{i\in I}$. Furthermore we have $S_{W,v}^{-1}(M_{\pm})=M_{\mp}^{[\perp]}$. So $\{S_{W,v}^{-1}(W_i)\}_{i\in I_+}$ is a collection of uniformly $J$-definite subspace of $M_-^{[\perp]}$. Similarly $\{S_{W,v}^{-1}(W_i)\}{i\in I_-}$ is a collection of uniformly $J$-definite subspace of $M_+^{[\perp]}$. Also $R(S_{W,v}^{-1}T_{{W,v}_+})$ and $R(S_{W,v}^{-1}T_{{W,v}_-})$ is maximal uniformly $J$-positive and $J$-negative subspace respectively. So from the definition of $J$-fusion frame we have that $\{(S_{W,v}^{-1}(W_i),v_i)\}_{i\in I}$ is a $J$-fusion frame for $\mathbb{K}$. $\{(S_{W,v}^{-1}(W_i),v_i)\}_{i\in I_{\pm}}$ is fusion frame for the Hilbert space $(M_{\mp}^{\perp},\pm[\cdot,\cdot])$. Now $(S_{W,v}^{-1}T_{W,v})(T_{W,v}^{\#}S_{W,v}^{-1})=S_{W,v}^{-1}$. Hence $S_{W,v}^{-1}$ is the required $J$-fusion frame operator. $\{(S_{W,v}^{-1}(W_i),v_i)\}_{i\in I}$ is called the cannonical $J$-dual fusion frame for $\{(W_i,v_i)\}_{i\in I}$ in $\mathbb{K}$.
\end{proof}
Let $(\mathbb{K},[\cdot,\cdot],J)$ be a Krein space and $\{f_i:i\in I\}$ be a $J$-frame for the Krein space $\mathbb{K}$. Then a careful investigation reveals that the sequence is not arbitrarily scattered in the Krein space. In fact the set of all positive elements form a maximal uniformly $J$-positive subspace $M_+=\overline{span\{f_i:i\in I_+\}}$ and the set of all negative elements form a maximal uniformly $J$-negative subspace $M_-=\overline{span\{f_i:i\in I_-\}}$. Now if we apply the $J$-frame operator $S^{-1}$ on the $J$-frame sequence then we know that the corresponding image set also decomposes the Krein space into two halves namely $M_+^{[\perp]}$ and $M_-^{[\perp]}$. So we have a nice distribution of the set $\{S^{-1}f_i:i\in I\}$. So in a rough sense we can say that the inverse of the $J$-frame operator \textit{i.e.} $S^{-1}$ rotates any uniformly $J$-definite subspace onto a uniformly $J$-definite subspace preserving the definiteness. Now let $-\infty<B_-\leq A_-<0<A_+\leq B_+<\infty$ be the optimal $J$-frame bounds for the $J$-frame $\{f_i:i\in I\}$. Now $\{S^{-1}f_i:i\in I\}$ is the cannonical $J$-dual frame for $\{f_i:i\in I\}$ in the Krein space $\mathbb{K}$. So the optimal frame bounds of this frame exists. We want to find a relation between the optimal bounds of the given $J$-frame and corresponding cannonical $J$-dual frame.
\begin{theorem}\label{BCJFF}
Let $\{f_i:i\in I\}$ be a $J$-frame for the Krein space $\mathbb{K}$ with optimal frame bounds $-\infty<B_-\leq A_-<0<A_+\leq B_+<\infty$. Then the cannonical $J$-dual frame has optimal frame bounds $-\infty<\frac{1}{A_-}\leq \frac{1}{B_-}<0<\frac{1}{B_+}\leq \frac{1}{A_+}<\infty$.
\end{theorem}
\begin{proof}
Let $S$ be the $J$-frame operator for the $J$-frame $\{f_i:i\in I\}$. Now consider the operator $S_+|_{M_+}$, it is a bijective, $J$-positive and $J$-selfadjoint. Also it is a frame operator for $\{f_i:i\in I_+\}$ in the Hilbert space $(M_+,[\cdot,\cdot])$. So, $A_+\leq S_+|_{M_+}\leq B_+$. Hence $\frac{1}{B_+}\leq (S_+|_{M_+})^{-1}\leq \frac{1}{A_+}$. But we know that $(S_+|_{M_+})^{-1}=S^{-1}_+|_{M_-^{[\perp]}}$. Hence from the definition of $J$-frame it easily follows that $\frac{1}{B_+}$ and $\frac{1}{A_+}$ are the optimal frame bounds of the frame $\{S^{-1}f_i:i\in I_+\}$. Similarly we can show that $\frac{1}{A_-}$ and $\frac{1}{B_-}$ are the optimal frame bounds of the frame $\{S^{-1}f_i:i\in I_-\}$. Hence we establish our result.
\end{proof}
\begin{theorem}
Let $\{f_i:i\in I\}$ be a $J$-frame for the Krein space $\mathbb{K}$ with cannonical $J$-dual frame $\{S^{-1}f_i:i\in I\}$. Then for all $I_1\subset I$ and for all $f\in \mathbb{K}$ we have
$$\sum_{i\in{I_1}}\sigma_i|[f,f_i]|^2-\sum_{i\in{I}}\sigma_i|[S_{I_1}f,S^{-1}f_i]|^2=\sum_{i\in{I_1^c}}\sigma_i|[f,f_i]|^2-\sum_{i\in{I}}\sigma_i|[S_{I_1^c}f,S^{-1}f_i]|^2$$, where $\sigma_i=1$ if $i\in{I_+}$ and $\sigma_i=-1$ if $i\in{I_-}$.
\end{theorem}
\begin{proof}
Let $S$ denote the frame operator for $\{f_i:i\in I\}$. Then we have $S(f)=\sum_{i\in I}\sigma_i[f,f_i]f_i$. Also $S=S_{I_1}+S_{I_1^c}$. Then $I=S^{-1}S_{I_1}+S^{-1}S_{I_1^c}$. From operator theory we have $S^{-1}S_{I_1}-S^{-1}S_{I_1^c}=S^{-1}S_{I_1}S^{-1}S_{I_1}-S^{-1}S_{I_1^c}S^{-1}S_{I_1^c}$. Then for every $f,g\in{\mathbb{K}}$ we have
\begin{equation*}
[S^{-1}S_{I_1}(f),g]-[S^{-1}S_{I_1}S^{-1}S_{I_1}(f),g]=[S_{I_1}(f),S^{-1}g]-[S^{-1}S_{I_1}(f),S_{I_1}S^{-1}g]
\end{equation*}
Now if we choose $g=S(f)$, then the above equation reduces to 
\begin{equation*}
=[S_{I_1}(f),f]-[S^{-1}S_{I_1}(f),S_{I_1}(f)]=\sum_{i\in{I_1}}\sigma_i|[f,f_i]|^2-\sum_{i\in{I}}\sigma_i|[S_{I_1}f,S^{-1}f_i]|^2
\end{equation*}
Now replacing $I_1$ by $I_1^c$ we can have the other part of the equality. Combining we finally get
\begin{equation*}
\sum_{i\in{I_1}}\sigma_i|[f,f_i]|^2-\sum_{i\in{I}}\sigma_i|[S_{I_1}f,S^{-1}f_i]|^2=\sum_{i\in{I_1^c}}\sigma_i|[f,f_i]|^2-\sum_{i\in{I}}\sigma_i|[S_{I_1^c}f,S^{-1}f_i]|^2
\end{equation*}
\end{proof}
Theorem (\ref{BCJFF}) can easily be generalized in the setting for $J$-fusion frame. We only state the result in the following theorem. Casazza et al. \cite{ck2003} calculated cannonical $J$-fusion frame bounds in a more general setting. But an error was pointed out by Gavruta \cite{pg}. But we calculate cannonical $J$-fusion frame bounds different from their approach.
\begin{theorem}
Let $\{(W_i,v_i):i\in I\}$ be a $J$-fusion frame for the Krein space $\mathbb{K}$ with optimal frame bounds $-\infty<B_-\leq A_-<0<A_+\leq B_+<\infty$. Then the cannonical $J$-dual fusion frame has optimal frame bounds $-\infty<\frac{1}{A_-}\leq \frac{1}{B_-}<0<\frac{1}{B_+}\leq \frac{1}{A_+}<\infty$.
\end{theorem}
The following lemma is the Krein space version of a theorem in \cite{pg}. We need this result in next part of our work.
\begin{lemma}
Let $T:\mathbb{K}\rightarrow \mathbb{K}$ be any bouned linear operator and $V$ be any closed regular subspace of $\mathbb{K}$. Then $Q_VT^{\#}=Q_VT^{\#}Q_{\overline{T(V)}}$, where $Q_V$ is the $J$-orthogonal projection onto $V$.
\end{lemma}
Now we want to address the problem of characterizing those bounded operators $T:\mathbb{K}\rightarrow \mathbb{K}$ such that $\{(T(W_i),v_i):i\in I\}$ is a $J$-fusion frame for $\mathbb{K}$ if $\{(W_i,v_i):i\in I\}$ is a $J$-fusion frame for $\mathbb{K}$. Now to form $J$-fusion frame the subspaces $T(W_i)$ must be uniformly definite. We now provide an example to show that the image of a closed, uniformly definite subspace under a bounded invertible linear operator may be neutral subspace.
\begin{example}
We will define an inner product $[\cdot,\cdot]$ on the sequence space $\ell^2$ in the following way. Let $\{e_n\}_{n\in\mathbb{N}}$ be the countable orthonormal basis. Let $[e_{2n},e_{2n}]=-1,~[e_{2n-1},e_{2n-1}]=1$ for all $n\in\mathbb{N}$ and also $[e_i,e_j]=0$ for $i\neq j$. Let $J:\ell^2\rightarrow\ell^2$ defined by $J(\sum_{n\in\mathbb{N}}c_ne_n)=(\sum_{n\in\mathbb{N}}\sigma_nc_ne_n)$, where $\sum_{n\in\mathbb{N}}c_ne_n\in\ell^2$ and $\sigma_n=1$ if $n$ is odd, $\sigma_n=1$ if $n$ is even. Then the triple $(\ell^2,[\cdot,\cdot],J)$ form a Krein space. Consider the invertible linear operator $T:\ell^2\rightarrow\ell^2$ defined by $T(\{c_n\}_{n\in\mathbb{N}})=(c_1+c_2,c_1+2c_2,c_3,\ldots)$. Now let $M=span\{e_1\}$. Then $M$ is a uniformly $J$-positive definite subspace. But $T(M)=span\{(1,1,0,\ldots)\}$ is a neutral subspace of $\ell^2$.
\end{example}
Now we will consider some restrictions on the linear operator $T$ so that $\{(T(W_i),v_i):i\in I\}$ also a $J$-fusion frame for $\mathbb{K}$. We will also calculate the corresponding $J$-fusion frame bounds.
\begin{definition}
Let $T$ be a bounded linear operator on a Krein space $\mathbb{K}$. $T$ preserves definiteness if $T(V)$ is definite whenever $V$ is definite where $V$ is a subspace of $\mathbb{K}$. We also say $T$ preserves definiteness with sign if the linear operator preserves definiteness and also the sign of the subspaces $V$ and $T(V)$ are same.
\end{definition}
\begin{definition}
Let $T$ be a bounded linear operator on a Krein space $\mathbb{K}$. $T$ preserves maximality if $T(V)$ is also maximal whenever $V$ is a maximal subspace of $\mathbb{K}$.
\end{definition}
\begin{definition}
Let $T$ be a bounded linear operator on a Krein space $\mathbb{K}$. $T$ preserves regularity if $T(V)$ is also regular subspace of $\mathbb{K}$ whenever $V$ is a regular subspace of $\mathbb{K}$.
\end{definition}
\begin{theorem}
Let $T$ be a bounded surjective linear operator on a Krein space $\mathbb{K}$. Also let\\
$(i)$ $T$ preserves definiteness with sign.\\
$(ii)$ $T$ preserves maximality.\\
$(iii)$ $T$ preserves regularity.\\
Then $\{(T(W_i),v_i):i\in I\}$ is a $J$-fusion frame for the Krein space $\mathbb{K}$ if $\{(W_i,v_i):i\in I\}$ be a $J$-fusion frame for $\mathbb{K}$.
\end{theorem}
\begin{proof}
Let $I_+=\{i\in{I}:[f_i,f_i]\geq 0~\textmd{for all~} f_i\in{W_i}\}$ and $I_-=\{i\in{I}:[f_i,f_i]<0~\textmd{for all~} f_i\in{W_i}\}$. For $i\in{I_+}$ choose $W_i$. Now each $W_i$ is closed, definite subspaces of $\mathbb{K}$. Since $T$ preserves definiteness with sign hence $T(W_i)$ is also positive definite for $i\in{I_+}$. $T(W_i)$ is also closed since the image of closed subspace is also closed as $T$ is bounded and linear. Now $M_+=\overline{\sum_{i\in{I_+}}W_i}$ is maximal uniformly $J$-positive subspace of $\mathbb{K}$. Now $T(M_+)\subset\overline{\sum_{i\in{I_+}}T(W_i)}$. By virtue of our assumptions $\overline{\sum_{i\in{I_+}}T(W_i)}$ is a positive subspace of $\mathbb{K}$. But since $T$ preserves maximality hence $T(M_+)=\overline{\sum_{i\in{I_+}}T(W_i)}$. Similarly for $i\in{I_-}$ we can show that $\overline{\sum_{i\in{I_-}}T(W_i)}=T(M_-)\subset\mathbb{K}$ is a maximal negative subspace of $\mathbb{K}$. Now we will use our regularity assumption. Since $T$ preserves regularity hence $T(M_+)$ and $T(M_-)$ are also regular. Using corollary 7.17 of \cite{ia} we have $T(M_+)$ and $T(M_-)$ are maximal uniformly $J$-positive and $J$-negative subspaces respectively. So we have a decomposition of $\mathbb{K}$ \textit{i.e.} $\mathbb{K}=T(M_+)\oplus T(M_-)$. Now let $\theta$ be the synthesis operator for the Bessel sequence of subspaces $\{W_i,v_i):i\in I\}$. Hence $\theta$ is surjective bounded linear operator. Then the mapping $T\theta$ is well defined and surjective. Now from the definition of $J$-fusion frame it easily follows that $\{(T(W_i),v_i):i\in I\}$ is also a $J$-fusion frame for the Krein space $\mathbb{K}$.
\end{proof}
\begin{remark}
Let the linear operator $T$ considered above is also injective. Then from \cite{jb} we know that $T$ is a scaler multiple of $J$-isometry. Therefore the class of operators are just $J$-unitary operators modulo multiplication by non-zero scalers.
\end{remark}
\begin{remark}
The conditions of the above theorem are sufficient but not necessary. In fact we can get a necessary conditions on $T$ which we thought worth mentioning. 
\end{remark}
\begin{theorem}
Let $\{(W_i,v_i):i\in I\}$ be a $J$-fusion frame for a Krein space $\mathbb{K}$ and $T$ be a bounded surjective linear operator on $\mathbb{K}$ such that $\{(T(W_i),v_i):i\in I\}$ is also a $J$-fusion frame for $\mathbb{K}$. Then $\exists$ index set $I^{0}$ such that $I^{0}=I$ and $\mathbb{K}=\overline{\sum_{i\in{I^0_+}}T(W_i)}\oplus\overline{\sum_{i\in{I^0_-}}T(W_i)}$ where $I^0_+\cup I^0_-=I^{0}$. Also $\overline{\sum_{i\in{I^0_+}}T(W_i)}$ and $\overline{\sum_{i\in{I^0_-}}T(W_i)}$ are maximal uniformly $J$-definite subspaces but of course with opposite signs.
\end{theorem}

{\bf Acknowledgments: } The author gratefully acknowledge the financial support of CSIR, Govt. of India.

\bibliographystyle{amsplain}

\end{document}